\documentclass{article}
\usepackage{amssymb,amsmath,latexsym, amsthm,enumerate,  amsfonts}
\usepackage{wrapfig, tikz}
\usetikzlibrary{matrix, calc}
\usepackage{blindtext}
\title{} \author{} \date{}
\usepackage[pdftex, pdfstartview=FitH]{hyperref} 

\numberwithin{equation}{section} 

\newcommand{\Cl}{\mathop{\mathrm {Cl}}\nolimits}
\newcommand{\Int}{\mathop{\mathrm {Int}}\nolimits}

 \newtheorem{thm}{Theorem}
 \newtheorem{cor}{Corollary}
 \newtheorem{lem}{Lemma}
 \newtheorem{prop}{Proposition}
 \theoremstyle{definition}
 \newtheorem{defn}{Definition}
  \newtheorem{ex}{Example}
 \theoremstyle{remark}
 \newtheorem{rem}{Remark}

\usepackage[OT2,OT1]{fontenc}
\newcommand\cyr
{
	\renewcommand\rmdefault{wncyr}
	\renewcommand\sfdefault{wncyss}
	\renewcommand\encodingdefault{OT2}
	\normalfont
	\selectfont
}
\DeclareTextFontCommand{\textcyr}{\cyr}

 \usepackage[paperwidth=210mm, paperheight=297mm, hmargin={20mm,20mm}, vmargin={20mm,20mm} ]{geometry}

\begin{document}
\thispagestyle{empty}
\begin{center}
{\large \bf On topology expansion using ideals

} \vspace*{3mm}

{\bf Anika Njamcul\footnote{Department of Mathematics and Informatics, Faculty of Sciences, University of Novi Sad, Serbia, \\ \href{https://orcid.org/0000-0001-5707-7074}{orcid.org/0000-0001-5707-7074}, e-mail: \href{mailto:anika.njamcul@dmi.uns.ac.rs}{anika.njamcul@dmi.uns.ac.rs}}}
and
{\bf Aleksandar Pavlovi\'c\footnote{Department of Mathematics and Informatics, Faculty of Sciences, University of @Novi Sad, Serbia, \\ \href{https://orcid.org/0000-0001-5001-6869}{orcid.org/0000-0001-5001-6869}, e-mail: \href{mailto:apavlovic@dmi.uns.ac.rs}{apavlovic@dmi.uns.ac.rs}}}
\end{center}

\begin{abstract}

 Topologies can be expanded with the help of ideals, using the local function, an operator resembling the closure of a set.

 The aim of this paper is  to  define  the ideals  which enable us to create  this topology $\tau^{*}$ on $X$ simultaneously making a specific set $A\subseteq X$ open in $\tau^{*}$. We study certain properties of $\tau^{*}$, especially under the assumption that $A$ is a preopen set. Further, we reflect on the ideal topological space in which the ideal is generated by a chosen family of dense sets. Here we prove that the generated topology by this ideal is submaximal, but not maximal connected.\\

 {\it AMS Mathematics  Subject Classification $(2020)$}:
54A10, 
54A05,  
54B99, 
54D05, 
54D35, 
54E99 
\\[1mm] {\it Key words and phrases:} ideal topological space, local function, dense sets, regular open sets, semiregularization, submaximal space,  connectedness, maximal topology

\end{abstract}

\section{Introduction}

The concept of the local function in general topology is well known since the work of Kuratovski (\cite{KURold}) in the first half of the twentieth century.
A more detailed study of the local function and its properties considering different ideals (the ideal of finite sets, codense sets, discrete sets, etc.) was given in 1944 by Vaidyanathaswamy  \cite{Vaid}.
The new topology gained by the local function was further studied by Hayashi \cite{Hay} in 1964 and  Nj\.{a}stad \cite{NJAST} in 1966, later by Samuels \cite{Sam} in 1975  and many others.
 For a period of time, this topic was mostly bypassed in research until the  paper by Jankovi\' c and Hamlett in 1990  \cite{JH}. In that paper, along with some new results,  they summarised all the known facts about this topic.

Thereafter, many papers concerning this subject were published. For example, Ekici defined I-Alexandroff ideal topological spaces in \cite{ERD}, while  Al-Omari and Noiri \cite{ON} in 2013  devoted their attention to  investigating the generalisation of $\theta$-closure by ideals, and they  defined the local closure function. Others (like Hatir \cite{HAT}) studied the properties of continuous  functions defined on ideal topological spaces.

Most of these approaches to the topic are similar - either the definition of the local function is adjusted, thus creating a new topology $\tau^{*} $ with various characteristics, or different more or less known ideals are considered together with the resulting topology.

In this paper we considered a less common approach.
 We start with the question of how to define an ideal so that the given set $A\subseteq X$ would become open in $\tau^{*}$. This is successfully done and it is shown that this method in some cases can not give minimal extension of a topology. We further prove that if the set $A$ in question is preopen, regular-open sets are preserved during this topology expansion.
 Also, in the special case of this construction, we obtain the result combining preopen sets and connectedness of the topologies $\tau$ and $\tau^{*}$.

 Since dense sets have given us interesting results when used in constructing ideals, in the second part of our work we consider the ideal obtained by a family of dense sets such that every finite intersection in this family is again dense in $\tau$. It turns out that this construction can give us important results concerning maximal topologies.

 In 1943 Hewitt \cite{HEW} studied maximal irresolvable (MI) spaces as spaces without isolated points in which every dense set is open. Amongst other things, he proved that every submaximal space  is an MI space. Later, submaximal spaces were studied as ones where every dense set is open (see \cite{ARH}).  Mioduszewski  and Rudolf  \cite{MR} investigated how regular open sets are preserved in topology expansions. They defined  r.o. maximal spaces and proved that a space is r.o. maximal if and only if every dense subset is open.

 We prove that the topology $\tau^{*}$ in the case of the ideal defined using the family of dense sets is submaximal, although not maximal connected.

 Maximal connected topology was principally investigated by Thomas in 1968. (see \cite{THO}) and many others thereafter.

\section{Preliminaries}

To begin with, let us state the notation we will use and some basic definitions and theorems.

  We will consider  $(X,\tau)$ to be  a topological space and $\tau(x)$  will denote the neighbourhood  system of the point $x\in X$. It is important to state that, for now, no additional properties of the topology are presumed.  The closure (resp. the interior) of a set $A\subseteq  X$ will be denoted by $\Cl(A)$ (resp. $\Int(A)$).

Considering  a nonempty family $\mathcal{I}$ of subsets of $X$,  $\mathcal{I}$ is an ideal on $X$ if the union of two sets from $\mathcal{I}$ is in $\mathcal{I}$ and if every subset of $A\in\mathcal{I}$ is again in $\mathcal{I}$.

If $(X,\tau)$ is a topological space and $\mathcal{I}$ an ideal on $X$, we say that $(X,\tau,\mathcal{I})$ is an ideal topological space.\\

A \textbf{local function} of  the set $A\subseteq X$ with respect to the topology $\tau$ and an ideal $\mathcal{I}$ both defined on $X$ is represented in the following way (see \cite{Veli}):

$$A^{*}(\tau,\mathcal{I})=\{x\in X:U\cap A\notin \mathcal{I}\mbox{ for every }U\in\tau(x)\}.$$

With the local function, a closure operator is defined as  $$\Cl^{*}(A)=A\cup A^{*},$$ and therefore a new topology $\tau^{*}$, as a topology with the closure operator $\Cl^*$,  is obtained.
  It is clear that  a set $F$ is closed in $\tau^{*}$ if and only if $F^*\subseteq F$, and, therefore, $U$ is open in $\tau^{*}$ if and  only if $U\subseteq X\setminus (X\setminus U)^{*}$ (see \cite{JH}). Also, let us notice that all sets in $\mathcal{I}$ are closed in $\tau^{*}$.

It can be shown  that the collection $$\beta(\mathcal{I},\tau)=\{V\setminus I:V\in\tau,I \in \mathcal{I}\}$$
 is a basis for $\tau^{*}$.

Working with ideal topological spaces, the notion of compatibility is of great significance.

\begin{defn}\cite{Vaid}
	Let $(X,\tau,\mathcal{I})$ be an ideal topological space. We say that topology $\tau$ is \textbf{compatible} with the ideal $\mathcal{I}$, denoted $\tau \sim \mathcal{I}$, if the following holds for every $A\subseteq X$: if for every $x\in A$ there exists a $U\in\tau(x)$ such that $U\cap A\in \mathcal{I}$, then $A\in\mathcal{I}$.
\end{defn}

Let us notice that if $\mathcal{I}$ is a principal ideal, then topology $\tau$ is trivially {compatible} with the ideal $\mathcal{I}$. Also, it is known (see \cite{JH}) that in the case of compatibility of a topology $\tau$ and an ideal $\mathcal{I}$, family $\beta(\mathcal{I},\tau)$ is not just a base of $\tau^*$, but equal to it.



If a topological space $(X,\tau)$ is given, a set $U\in\tau$ is said to be regular open if $U=\Int(\Cl(U))$. It is well known that the collection of all regular open sets in $\tau$  forms a basis for a topology $\tau_s$ on $X$, which is also called the \textbf{semiregularization} of $\tau$. If we consider  all topologies on $X$ with the same family of regular open sets, it is said that these topologies are r.o. equivalent \cite{MR}. Also, a topology $\tau$ on $X$ is said to be r.o. maximal \cite{MR} if for every r.o. equivalent topology $\tau'$ on $X$ such that $\tau\subseteq \tau'$ we have $\tau=\tau'$.

Recall that a  topological space $(X,\tau)$ is connected if and only if it can not be represented as the union of two non-empty, disjoint open subsets.

A set $A\subseteq X$ is preopen (see \cite{GAN}, \cite{DONT}) if and only if $A\subseteq \Int(\Cl(A))$. A  topological space $(X,\tau)$ is \textbf{submaximal} iff every preopen set is open. It is known (see \cite[p.139]{Bourbaki}) that submaximal spaces are r.o. maximal and vice versa.

\section{ Simple topology expansion }
Firstly, we will address the following problem: if a topological space $(X,\tau)$ and an arbitrary set $A\subseteq X$ are given, how can we define an ideal $\mathcal{I}_A$ such that $A$ is open in the  topological space $\langle X, \tau^{*}\rangle$ generated by $\mathcal{I}_A$?

There are always two  easy solutions for this problem. The first one is to take $P(X)$ for a required ideal, but then we obtain discrete topology. The other solution is to take  $\{I: I \subseteq X \setminus A\}$ for a required  ideal, that is the principal ideal obtained by $X \setminus A$. But this solution makes every subset of $X\setminus A$ closed, that is, every superset of $A$ open. Both solutions are much larger than   the  simplest topology containing $\tau$, and $A$, the topology with the subbase $\tau \cup \{A\}$,  usually denoted by $\tau(A)$, which is, in the lattice of topological spaces, the   supremum of $\tau$ and atomic topology $\{\emptyset, A, X\}$ (see \cite{Cam71}).

So, is there a "small" solution obtained by ideal?

It is necessary to assume that  the set $A$ in question is not open in $\tau $.  Obviously, if $A$ is open in $\tau$, by definition, $A$ is also open in $\tau^{*}$, and the required ideal is $\{\emptyset\}$. Therefore, in this section we will assume that $A\notin \tau$.

In order to define the ideal which will "make $A$ open" in $\tau^{*}$, for each $x\in A$ we will take its arbitrary neighbourhood $U_x$ (van Douwen called this "a neighbourhood assignment" in \cite{vDPff}).

 The family of sets $$\{U_x\setminus A:x\in A\}$$ will generate the ideal $\mathcal{I}_A$ on $X$ by adding to it   all the subsets of the sets in the given family and their finite unions.

 Note that if $x\in \Int(A)$, then there exists $U\in\tau(x)$ such that $U\setminus A=\emptyset$. This is why  we can instead use   the ideal generated by the family  $$\{U_{x}\setminus A:x\in A\setminus \Int(A)\}.$$

 Let us prove that this is one answer to our problem.

 \begin{prop}
 	If $(X,\tau, \mathcal{I}_A)$ is an ideal topological space, then $A\in\tau^{*}$.
 \end{prop}
 \begin{proof}
 	We will prove that the complement of $A$ is closed in $\tau^{*}$ by showing that  $(X\setminus A)^{*}\subseteq X\setminus A$. Let  $x\in (X\setminus A)^{*}$. By definition, for every $O\in \tau(x)$ is $O\cap (X\setminus A)\notin \mathcal{I}_A$ or, equivalently, $O\setminus A\notin\mathcal{I}_A$. Assume $x\in A$. If $x\in \Int(A)$,  then $\Int(A)\setminus A=\emptyset\in\mathcal{I}_A$. On the other hand, if $x\in A \setminus \Int(A)$, then for any $O\in\tau(x)$ we have $(O\cap U_x)\setminus A\subset U_x \setminus A\in\mathcal{I}_{A}$, which contradicts $x\in (X\setminus A)^{*}$. Thus, $x\in X\setminus A$.
 \end{proof}

 Therefore, using the ideal $\mathcal{I}_A$ we have constructed an ideal topological space in which  the chosen set $A$ is open.

%



	


We will illustrate the construction of $\mathcal{I}_A$ with the following example.

 \begin{ex}
 Observe the set of real numbers $\mathbb{R}$ with the right-ray  topology, that is,  $\tau=\{(a,+\infty):a\in\mathbb{R}\}\cup\{\emptyset,\mathbb{R}\}.$ Take $\mathbb{N}$ to be the set of positive integers. Clearly, $\mathbb{N}$ is dense in $\tau$. Since $\mathbb{N}\setminus \Int(\mathbb{N})=\mathbb{N}$, let $U_n=(n-1,\infty)$. The ideal $\mathcal{I}_{\mathbb{N}}$ generated with $\mathbb{N}$ consists of sets from the family $\{(n-1,+\infty)\setminus \mathbb{N}:n\in \mathbb{N}\}$ and their subsets and finite unions. It is easy to see that $\mathcal{I}_{\mathbb{N}}=P((0,\infty) \setminus \mathbb{N})$.
 Since sets $V \setminus I$, where $V\in \tau$ and $I \in \mathcal{I}_\mathbb{N}$ are sets from the base of $\tau^*$, it is easy to see that $(0,\infty) \setminus ((0,\infty) \setminus \mathbb{N})=\mathbb{N}$ is open in $\tau^*$.

\end{ex}

We will note the following lemma, necessary for the proof of our next result.

\begin{lem}\label{L5} Assume $(X,\tau,\mathcal{I}_{A})$ is an ideal topological space, with $A\subseteq X$ and  $x_{1},x_{2}\in A\setminus \Int(A)$, $U_{x_{1}}\in\tau(x_{1})$, $U_{x_{2}}\in\tau(x_{2})$. If $\Int(U_{x_{1}}\setminus A)=\emptyset$  and $\Int(U_{x_{2}}\setminus A)=\emptyset$, then $
\Int((U_{x_{1}}\setminus A)\cup (U_{x_{2}}\setminus A))=\emptyset$.
\end{lem}
\begin{proof}
	If we take a nonempty open set $O$ in $\tau$ such that $O\subseteq (U_{x_{1}}\setminus A)\cup (U_{x_{2}}\setminus A)$, then $O$ is an open set in $\tau$ disjoint with $A$. Also, we have either $O\cap (U_{x_{1}}\setminus A)\neq \emptyset$ or $O\cap (U_{x_{2}}\setminus A)\neq \emptyset$. Suppose, without loss of generality, that $O\cap (U_{x_{1}}\setminus A)\neq \emptyset$. Then $\emptyset=O \setminus A \supseteq (O\cap U_{x_1}) \setminus A=O\cap (U_{x_1} \setminus A)\neq \emptyset$. A contradiction.
\end{proof}

Let us note that an elementary proof by induction will show that this lemma can be applied for all such  finite unions.

\begin{ex}
	
	In general, if the interior of two sets is empty, the interior of their union need not also be empty. For example, take the set of real numbers $\mathbb{R}$ with the usual topology. The set of rational numbers $\mathbb{Q}$ has an empty interior in $\mathbb{R}$, so does the set of irrational numbers, $\mathbb{R}\setminus \mathbb{Q}$, but $\Int(\mathbb{Q}\cup (\mathbb{R}\setminus \mathbb{Q}))=\mathbb{R}$.
	
\end{ex}

The next Lemma directly follows from Theorem 6.1. in \cite{JH}. We will  give an alternate proof using Lemma \ref{L5}.

\begin{lem}\label{L1}
	Let $(X,\tau,\mathcal{I}_A)$ be an ideal topological space.  Then $\tau\cap \mathcal{I}_A=\{\emptyset\}$ if and only if for every $x\in A\setminus \Int(A)$, we have  $\Int(U_x\setminus A)=\emptyset$.
\end{lem}
\begin{proof}
	Assume $\tau\cap \mathcal{I}_A=\{\emptyset\}$. Since  for all   $x\in A\setminus \Int(A)$  we have $\Int(U_{x}\setminus A)\subseteq U_{x}\setminus A\in\mathcal{I}_A$, we have   $\Int(U_{x}\setminus A)\in \tau \cap \mathcal{I}_A$, implying $\Int(U_{x}\setminus A)=\emptyset$.
	
	Now assume  that, for every $x\in A\setminus \Int(A)$, we have  $\Int(U_x\setminus A)=\emptyset$. Take $O\in\tau\cap \mathcal{I}_A$. Then, since $O$ is in the ideal $\mathcal{I}_A$, it is a subset of finite unions of the sets of the form $U_{x}\setminus A$. So we have $O\subseteq \bigcup_{i=1}^{k}(U_{x_{i}}\setminus A)$, where $U_{x_{i}}\in\tau(x_{i})$, $x_{i}\in A\setminus \Int(A),i=1,2,\cdots, k$. Since $O$ is open, we have $O\subseteq \Int(\bigcup_{i=1}^{k}(U_{x_{i}}\setminus A))$ and by Lemma \ref{L5},  $O $ has to be an empty set.
	
\end{proof}

By example, we will show that in some instances  $\tau \cap \mathcal{I}_{A}= \{\emptyset\}$ can fail.

\begin{ex}
    Let us consider the topological space $(\mathbb{R},\tau)$, where $\mathbb{R}$ is the set of real numbers and $\tau$ is the usual topology on $\mathbb{R}$ generated by intervals. Assuming $A=[0,1]$, we will generate $\mathcal{I}_A$ with the family of sets $\{U_x\setminus [0,1]:x\in  [0,1]\setminus (0,1)\}=\{U_x\setminus [0,1]:x\in \{0,1\}\}$, where  $U_x$ is an assigned arbitrary neighbourhood of $x$. Clearly, now we can see that there are sets in $\tau$ which are also in the ideal $\mathcal{I}_A$. For example, if we take for  $x=0$ a neighbourhood  $U_x=(-\varepsilon,\varepsilon)$, where $\varepsilon<1$, the set $(-\varepsilon,0)$ is open and will be in $\mathcal{I}_A$.
\end{ex}

In order to obtain as less as we can new open sets, the idea is to take a smaller ideal. Our next result will illustrate how assuming a certain separation axiom on $(X,\tau)$ can effect the resulting ideal topological space.

\begin{thm}
	Let  $(X,\tau)$ be a $T_{1}$-space. If $A \not \in \tau$, then there is no minimal ideal which generates in the ideal topological space $(X,\tau,\mathcal{I}_{A})$ topology $\tau^*$ which contains $A$.
\end{thm}
\begin{proof}
	If $(X,\tau,\mathcal{I}_{A}) $ is given, we will show that there is always an ideal $\mathcal{I}^{'}_A \subsetneq \mathcal{I}_A$ obtained in the same way using different van Douwen assignment.

	Take an arbitrary $y\in U_{x_0}\setminus A$, where $x_0$ is an arbitrary element of $A\setminus \Int(A)$. Such $y$ always exists, since, otherwise, $x_0$ would be in $\Int(A)$. So, $\{y\}\in \mathcal{I}_{A}$. Since $X$ is a $T_{1}$-space, $\{y\}$ is a closed set, so $\widetilde{U_{x}}=U_{x}\setminus \{y\}$ is an open neighbourhood of a point $x$ from $A\setminus \Int(A)$ which does not contain $y$. Let $\widetilde{\mathcal{I}_A}$ be the ideal generated by the family $\{\widetilde{U_{x}}\setminus A:x\in A\setminus \Int(A)\}$.  By construction, $y\notin I$ for all $I\in \widetilde{\mathcal{I}_A}$ and $\widetilde{\mathcal{I}_A}\subsetneq \mathcal{I}_{A}$. In this way we found an ideal "strictly smaller" than $\mathcal{I}_{A}$ such that the  ideal topological space $(X,\tau,\widetilde{\mathcal{I}_A}\rangle$ generates a topology in which $A$ is open.
\end{proof}

\section{Preserving regular open sets and connectivity}

 Obviously, $\tau^{*}$ is an extension of the topology $\tau$ on $X$.  The next proposition derives from the example of Mioduszewski and Rudolf in \cite{MR}, in which they show that every extension with a dense set preserves regular open sets.

To prove that, we need a following theorem.

\begin{thm}\cite[Th. 6.4]{JH} \label{T64JH} If $\tau \cap \mathcal{I}=\{\emptyset\}$ then $\tau_s=(\tau^{*})_s$.
\end{thm}

 \begin{prop}\label{prop2}
 	Let $(X,\tau,\mathcal{I}_A)$ be an ideal topological space  If $A\subseteq X$ is a dense set in $\tau$, then regular open sets in $\tau$ are also regular open in $\tau^*$, i.e.\ $\tau_s=(\tau^{*})_s$.
 \end{prop}
 \begin{proof}
 	For an arbitrary $O \in \tau$, $\Int(O\setminus A)$ is an empty set, otherwise, $A$ would not be dense in $\tau$. Therefore, by Lemma \ref{L1} we have $\tau \cap \mathcal{I}_A =\{\emptyset\}$, which, together with Theorem \ref{T64JH}, proves the statement.
 \end{proof}

 Note that the condition in the previous proposition need not to be necessary for the preservation of regular open sets, as shown by the following example.
 \begin{ex}
 	Observe the space $X=(2,3)\cup (4,5)$ with the usual open interval topology $\tau$ and let $ A=(2,3)\cap \mathbb{Q}$, where $\mathbb{Q}$ is the set of rational numbers. Obviously, $A$ is not dense in $X$, but regular open sets in $\tau$ are regular open in $\tau^{*}$.
 \end{ex}

Now we will study the connectedness of the topological space $(X,\tau^{*})$ generated by $\mathcal{I}_A$. In this context, the set $A$ is necessarily a preopen set, as we show in the following proposition.

\begin{prop}\label{prop3}
	If $(X,\tau,\mathcal{I}_{A})$ is an ideal topological space and $A\subseteq X$ is not preopen in $\tau$, then $\tau^{*}$ is not connected.
\end{prop}
\begin{proof}
	Assume that $A$ is not preopen in $X$. Then, there exists $x \in A\setminus (\Int(\Cl(A)))$. So, $U_x \setminus \Cl(A)  \neq \emptyset$, and, since  $U_x \setminus A \in  \mathcal{I}_{A}$, we have $U_x \setminus \Cl(A) \in \tau \cap \mathcal{I}_{A}$. By Lemma 2.7 in \cite{JH}, each set in ideal is closed  in $\tau^{*}$, so $U_x \setminus \Cl(A)$ is nonempty clopen set in $\tau^{*}$. That means that $(X,\tau^{*})$ is not connected.
\end{proof}

\begin{rem}\label{R1}
	Also, note that if $\tau\cap \mathcal{I}_{A}\neq \{\emptyset\}$, then $(X,\tau^{*})$ is also not connected, since there is a set $B$ open in $\tau$ (and therefore in $\tau^{*}$) which is at the same time $\tau^{*}$-closed, because $B\in\mathcal{I}_{A}$.
\end{rem}

Let us now modify the ideal $\mathcal{I}_A$ as follows.

\begin{defn}
	Define the ideal $\mathcal{I}^{'}_A$ as the ideal generated with the family of sets $\{U_x\setminus A:x\in A\}$, by making van Douwen assignment in the following way:
\begin{enumerate} \item if exists, let $U_x$ be a minimal  neighbourhood of $x\in X$;
\item if not, then we take, if possible $U_x\in\tau(x)$ such that $U_x\subseteq \Int(\Cl(A))$;
\item if not, then we choose $U_x$ arbitrary.
\end{enumerate}
	\end{defn}

\begin{lem}\label{L2}
	Let $(X,\tau,\mathcal{I}^{'}_{A})$ be an ideal topological space.  Then $\tau \cap \mathcal{I}^{'}_{A}=\{\emptyset\}$ if and only if $A$ is a preopen set.
\end{lem}

\begin{proof}
	Firstly, we will assume that $\tau \cap \mathcal{I}^{'}_{A}=\{\emptyset\}$ and that $A$ is not a preopen set, i.e. $A\nsubseteq \Int(\Cl(A))$. Therefore, there exists $x\in A$ such that $x\notin \Int(\Cl(A))$, thus for every neighbourhood $V_x$ of $x$ we have $V_x\setminus \Cl(A)\neq \emptyset$ and, consequently, $V_x\setminus A\neq \emptyset$. However, this is in contradiction with Lemma \ref{L1}, since now we have $\Int(V_x\setminus A)\supseteq \Int(V_x\setminus \Cl(A))\neq \emptyset$, which gives  $\Int(U_x\setminus A)\in (\tau\cap \mathcal{I}^{'}_A)\setminus \{\emptyset\}$, where $U_x$ is the van Douwen neighbourhood assignment for $x$. A contradiction.
	
	Now let us assume that $A$ is a preopen set. If $\tau \cap \mathcal{I}^{'}_{A}\neq \{\emptyset\}$, then there exists a non-empty set $V\in \tau\cap \mathcal{I}^{'}_A$. This also means that there exists $x\in A$ and $U_x\in\tau(x)$ such that $V\subseteq U_x\setminus A$, and therefore $\Int(U_x\setminus A)\neq \emptyset$. If  $U_x$ is chosen as a  minimal neighbourhood of $x$, then, since $A$ is preopen, $A \subseteq \Int(\Cl(A))$, implying $\Int(\Cl(A))$  is an neighbourhood of $A$. So, due to minimality of $U_x$,  $U_x \subseteq \Int(\Cl(A))$. Since $V$ is an open set outside of $A$, then $\Cl(A) \cap V=\emptyset$. So, $V \cap \Int(\Cl(A))=\emptyset$, implying $V \cap U_x=\emptyset$.    This contradicts the fact that $V \subseteq U_x\setminus A \subseteq U_x$. On the other hand, if we can not choose minimal $U_x$, since $A$ is preopen, we can always choose $U_x$ such that  $U_x\subseteq \Int(\Cl(A))$. Then $ V\subseteq U_x \setminus A\subseteq \Int(\Cl(A))\setminus A$. Since $A$ is a preopen set, this means that $\Int(\Int(\Cl(A))\setminus A)=\emptyset$, so $V$ is a subset of the empty set, which leads to a contradiction.
 \end{proof}

So, with this, we can strengthen Proposition \ref{prop2}.

\begin{cor}\label{cor_preopen}
 	Let $(X,\tau,\mathcal{I}_A)$ be an ideal topological space  If $A\subseteq X$ is a preopen  set in $\tau$, then regular open sets in $\tau$ are also regular open in $\tau^*$, i.e.\ $\tau_s=(\tau^{*})_s$.
 \end{cor}

The assumption that $A$ is preopen is also relevant  in studying the connectedness of the new topology.

\begin{thm}\label{thm_con_preopen}
	If $\tau$ is connected in the ideal topological space $(X,\tau,\mathcal{I}^{'}_{A})$, then $\tau^{*}$ is connected if and only if $A$ is a preopen set.
\end{thm}
\begin{proof}
	 If $\tau^{*} $  is connected,  from Proposition \ref{prop3}  we have that the set $A$ is preopen.\\
	
	On the other hand, let $A$ be a preopen set and assume $\tau$ is a connected topology, but $\tau^{*}$ is not. We well consider specially obtained topology $\tau_{max}^{*}$, obtained by van Douwen neighbourhood system such that $U^{max}_x=\Int(\Cl(A))$. Since $U_x  \subseteq U^{max}_x$,  we have $\mathcal{I}^{'}_{A} \subseteq \mathcal{I}^{max}_{A}$. Therefore $\tau^{*}\subseteq\tau_{max}^{*}$. So $\langle X, \tau_{max}^{*}\rangle$ is also disconnected space. But we have more. Since $\mathcal{I}^{max}_{A}$ is a principal ideal, we have compatibility of the topology and the ideal, i.e.\  $\tau \sim \mathcal{I}^{max}_{A}$.  This gives that $\beta(\mathcal{I}^{max}_{A}, \tau)=\tau_{max}^{*}$. So, there exists a set $B$  in $\tau_{max}^{*}$, but not in $\tau$ such that $X\setminus B\in\tau_{max}^{*}$. So, there exist $V_1, V_2\in \tau$ and $I_1,I_2\in\mathcal{I}^{max}_A$ so that $B=V_1\setminus I_1$ and $X\setminus B=V_2\setminus I_2$. So, $X=(X \setminus V_1) \cup (X \setminus V_2) \cup (I_1 \cup I_2)$, and those three sets are disjoint.  This gives that $I_1 \cup I_2$ is open, but, $I_1 \cup I_2 \subseteq \Int(\Cl(A))\setminus A=\emptyset$. This implies that $I_1 \cup I_2 = \emptyset$, which implies that $B=V_1$ and  $X \setminus B= V_2$, implying $\tau$ is  a disconnected topology. A contradiction.
 \end{proof}

\section{Topology generated with dense sets}

 In the second part of out paper we will demonstrate how using a specific family of dense sets one can generate an ideal topological space with interesting properties.\\

 \begin{defn}
 	In the topological space $(X,\tau)$  let   $\mathcal{D}$ be the family  of dense sets. We define a  subfamily $\mathcal{D}'\subseteq \mathcal{D}$ with the following property: for every finite subfamily $\overline{\mathcal{D}}$ in $\mathcal{D}'$ the intersection $\bigcap \overline{\mathcal{D}}$ is dense in $X$. We will say that $\mathcal{D}'$ has \textbf{the dense finite intersection property}.
 \end{defn}

Observe  that a maximal subfamily $\mathcal{D}'$ with the given property can be taken: for a chain $\mathcal{L}$ in $\mathcal{D}'$, obviously $\bigcup \mathcal{L}$ is a dense set in $X$ and an upper bound for $\mathcal{L}$. Therefore, simply by applying the Zorn's Lemma we get the maximal family with the dense finite intersection property.

However, this maximal subfamily of dense sets  need not be unique, as it is illustrated in the following example using resolvable spaces.

\begin{defn}[\cite{GAN}]
	A topological space $(X,\tau)$ is resolvable if it contains two disjoint dense subsets.
\end{defn}

\begin{ex}
	Let $(X,\tau)$ be a resolvable space with dense, disjoint subsets $A$ and $B$. It is obvious that only one of the sets $A$ and $B$ can be in the maximal family $\mathcal{D}'$ of dense subsets with the dense intersection property. Thus, we will have two maximal families - one containing $A$ and another containing  the set $B$.
\end{ex}

Using the previously described family  of dense sets $\mathcal{D}'$ we will define a new ideal on $X$. Precisely, we will consider the ideal $\mathcal{I}_D$ generated by  the family of sets $\{C: C\subseteq (X\setminus D)\}$ for all $D\in\mathcal{D'}$, that is, $\mathcal{I}_D$ will consist of the sets $\{C: C\subseteq (X\setminus D)\}$ for all $D\in\mathcal{D'}$, their subsets and their finite unions.\\

Unlike in our previous section, the ideal generated in this way will not contain sets open in $(X,\tau)$.

\begin{lem}\label{L3}  If  $(X,\tau)$ is an arbitrary topological space, then $\tau\cap \mathcal{I}_{D}=\{\emptyset\}$.
\end{lem}

\begin{proof}
	It suffices to notice that if $B$ is a non-empty open set, then it has to intersect every dense set, and therefore can not be in the ideal $\mathcal{I}_D$.
\end{proof}

\begin{ex}
	Observe the right-ray topology $\tau $ on the set of real numbers $\mathbb{R}$, that is,  $\tau=\{(a,\infty):a\in\mathbb{R}\}$. The set of natural numbers, $\mathbb{N}$ is dense in $\mathbb{R}$. Consider the ultrafilter generated by the cofinite sets in $\mathbb{N}$, that is, the sets with finite complements. Clearly, all finite intersections of these sets are again dense in $\tau$, so we can take this to be the family y $\mathcal{D}'$ described above.
\end{ex}

Also, topology extension using the ideal $\mathcal{I}_D$ preserves regular open sets.

\begin{prop}
	If $(X,\tau,\mathcal{I}_D)$ is an ideal topological space, then $\tau_s=(\tau^{*})_s$.
\end{prop}

\begin{proof}
	By our Lemma \ref{L3} and the result given in Theorem 6.1  in \cite{JH} by Jankovi\'c and Hamlett we see that $\tau\cap\mathcal{I}_D=\{\emptyset\}$ is equivalent to $X=X^{*}$, and from  the Theorem 6.4 in the same paper we have that the semiregularisations of $\tau$ and $\tau^{*}$ are equal.
	
\end{proof}

Thus, semiregularizations of $\tau$ and $\tau^{*}$ are the same. We will see in the sequel that the topology $\tau^{*}$ formed by the ideal $\mathcal{I}_D$ is in fact the r.o. maximal topology on $X$.

\begin{thm} If $(X,\tau,\mathcal{I}_D)$ is an ideal topological space, then $\tau^{*}$ is a r.o. maximal topology on $X$ (i.e.\ it is a submaximal space).
\end{thm}
\begin{proof}
	According to \cite[Theorem 1.2]{ARH} it is sufficient to prove that every dense set in $\tau^{*}$ is also open in $\tau^{*}$.  Let  $D_0$ be a  dense set in $\tau^{*}$ and assume it is not open in $\tau^{*}$. Since $D_0$ is dense in $\tau^{*}$ and not in $\mathcal{D}'$, there is a dense set $D_1$ such that $D_0\cap D_1$ is not dense in $X$. Thus there is an open set $O$ such that $O\cap (D_0\cap D_1)=\emptyset$. If we notice that $O\cap D_1$ is open in $\tau^{*}$, we can see that this would mean that $D_0$ is not dense in $\tau^{*}$.
\end{proof}

\section{Remark on maximal connected spaces}


Unfortunately, using this method of ideals, in most of the cases we will not obtain maximal connected topology. By \cite[Theorem 1]{CL}, space is maximal connected iff it is nearly maximal connected and every dense set is open. The problem is that the property of nearly maximal connectedness only depends on regular-open sets, which implies that if starting topology $\tau$ is not nearly maximal connected, than also obtained $\tau^*$ will not have this property. If we want to add new regular open sets, according to Theorem \ref{cor_preopen}, we will have to add some sets which are not preopen, but, in that case, by Theorem \ref{thm_con_preopen}, we will lose connectivity.

So, to obtain maximal connected space with this method, we have to start with nearly maximal connected space. Example of such space are maximal singular expansions of real line, given in \cite{GSW}. On the other hand (see \cite{CL}), natural topology on the real line is not nearly maximal connected, and therefore, by this method it can not be extended up to the maximal connected topology.

\section{Acknowledgments}
This talk is supported by the Science Fund of the Republic
of Serbia, Grant No. 7750027: Set-theoretic, model-theoretic and Ramsey-
theoretic phenomena in mathematical structures: similarity and diversity –
SMART.

The first author acknowledges financial support of the Ministry of Education, Science and Technological Development of the Republic of
Serbia (Grant No.\ 451-03-68/2020-14/ 200125).

The authors wishes to thank Adam Bartoš for a help provided in understanding nature of nearly maximal and maximal connected spaces.

\end{document}